\documentclass[12pt]{amsart}

\usepackage{tikz-cd}
\usepackage[export]{adjustbox}
\usepackage{amsrefs}
\usepackage{amssymb}
\usepackage{fancyhdr}
\usepackage{bbm}
\usepackage{xcolor}
\usepackage{hyperref}
\hypersetup{
  colorlinks=true,
  linkcolor=blue,
  citecolor=red
}

\newtheorem{theorem}{Theorem}
\newtheorem*{theorem*}{Theorem}
\newtheorem{lemma}[theorem]{Lemma}
\newtheorem{proposition}[theorem]{Proposition}
\newtheorem{claim}[theorem]{Claim}

\newtheorem{maintheorem}{Theorem}

\theoremstyle{definition}
\newtheorem{definition}[theorem]{Definition}
\newtheorem*{definition*}{Definition}
\newtheorem*{lemma*}{Lemma}
\usepackage{graphicx}
\usepackage{pstricks, enumerate, pst-node, pst-text, pst-plot}

\numberwithin{equation}{section}
\numberwithin{theorem}{section}

\newcommand{\R}{\mathbb{R}}
\newcommand{\N}{\mathbb{N}}

\newcommand{\Z}{\mathbb{Z}}

\newcommand{\eps}{\varepsilon}

\usepackage{xparse}
\DeclareDocumentCommand\Pr{ m g }{%
    {   \IfNoValueTF {#2}
      {\mathbb{P}\left[{#1}\right]}
      {\mathbb{P}\left[{#1}\middle\vert{#2}\right]}%
    }
}
\DeclareDocumentCommand\E{ m g }{%
    {   \IfNoValueTF {#2}
      {\mathbb{E}\left[{#1}\right]}
      {\mathbb{E}\left[{#1}\middle\vert{#2}\right]}%
    }
}

\DeclareMathOperator{\supp}{supp}

\def\dd{\mathrm{d}}

\usepackage{mathtools}
\newcommand{\norm}[1]{\left\lVert#1\right\rVert}

\newcommand{\cent}[1]{C_G(#1)}

\begin{document}

\title[]{Choquet-Deny groups and the infinite conjugacy class property}

\author[J.\ Frisch]{Joshua Frisch}
\author[Y.\ Hartman]{Yair Hartman}
\author[O.\ Tamuz]{Omer Tamuz}
\author[P.\ Vahidi Ferdowsi]{Pooya Vahidi Ferdowsi}

\address[J.\ Frisch, O.\ Tamuz, P.\ Vahidi Ferdowsi]{California Institute of Technology}
\address[Y.\ Hartman]{Northwestern University} 


\date{\today}

\begin{abstract}
A countable discrete group $G$ is called Choquet-Deny if for every non-degenerate probability measure $\mu$ on $G$ it holds that all bounded $\mu$-harmonic functions are constant. We show that a finitely generated group $G$ is Choquet-Deny if and only if it is virtually nilpotent. For general countable discrete groups, we show that $G$ is Choquet-Deny if and only if none of its quotients has the infinite conjugacy class property. Moreover, when $G$ is not Choquet-Deny, then this is witnessed by a symmetric, finite entropy, non-degenerate measure.
\end{abstract}

\maketitle

\section{Introduction}
Let $G$ be a countable discrete group. A probability measure $\mu$ on $G$ is {\em non-degenerate} if its support generates $G$ as a semigroup.\footnote{In the context of Markov chains such measures are called {\em irreducible}.} A function $f \colon G \to \R$ is {\em $\mu$-harmonic} if $f(k) = \sum_{g \in G}\mu(g) f(k g)$ for all $k \in G$. We say that the {\em measured group} $(G,\mu)$ is {\em Liouville} if all the bounded $\mu$-harmonic functions are constant; this is equivalent to the triviality of the Poisson boundary $\Pi(G,\mu)$~\cite{furstenberg1963noncommuting,furstenberg1971random,furstenberg1973boundary} (also called the Furstenberg-Poisson boundary; for formal definitions see also, e.g., Furstenberg and Glasner~\cite{furstenberg2009stationary}, Bader and Shalom~\cite{bader2006factor}, or a survey by Furman~\cite{furman2002random}).

When $G$ is non-amenable, $(G,\mu)$ is not Liouville for every non-degenerate $\mu$~\cite{furstenberg1973boundary}. Conversely, when $G$ is amenable, then there exists some non-degenerate $\mu$ such that $(G,\mu)$ is Liouville, as shown by Kaimanovich and Vershik~\cite{kaimanovich1983random} and Rosenblatt~\cite{rosenblatt1981ergodic}. It is natural to ask for which groups $G$ it holds that $(G,\mu)$ is Liouville for {\em every} non-degenerate $\mu$. We call such groups {\em Choquet-Deny} groups; as we discuss in \S\ref{sec:CD-defs}, there are a few variants of this definition (see, e.g.,~\cite{guivarc1973croissance, glasner1976proximal, glasner1976choquet}, or~\cite{jaworski2007choquet}), which, however, we show to be equivalent.

The classical Choquet-Deny Theorem (which was first proved for $\Z^d$ by Blackwell~\cite{blackwell1955transient}) states that abelian groups are Choquet-Deny~\cite{cd1960}; the same holds for virtually nilpotent groups~\cite{dynkin1961random}. There are many examples of amenable groups that are not Choquet-Deny:  first examples of such groups\footnote{In the Lie group setting, an example of an amenable group that is not Choquet-Deny was already known to Furstenberg~\cite{furstenberg1963noncommuting}.} are due to Kaimanovich~\cite{kaimanovich1983examples} and Kaimanovich and Vershik~\cite{kaimanovich1983random}, they include locally finite groups; Erschler shows that finitely generated solvable groups that are not virtually nilpotent are not Choquet-Deny~\cite{erschler2004liouville}, and that even some groups of intermediate growth are not Choquet-Deny~\cite{erschler2004boundary}.  Kaimanovich and Vershik~\cite[p. 466]{kaimanovich1983random} conjecture that ``Given an exponential group G, there exists a symmetric (nonfinitary, in general) measure with non-trivial boundary.'' See Bartholdi and Erschler~\cite{bartholdi2017poisson} for additional related results and further references and discussion.

Our main result is a characterization of Choquet-Deny groups.  We say that $G$ has the {\em infinite conjugacy class} property (ICC) if it is non-trivial, and if each of its non-trivial elements has an infinite conjugacy class. We say that $\mu$ is {\em fully supported} if $\supp \mu=G$; obviously this implies that $\mu$ is non-degenerate.
\begin{maintheorem}
\label{thm:main}
  A countable discrete group $G$ is Choquet-Deny if and only if it has no ICC quotients. Moreover, when $G$ does have an ICC quotient, then there exists a fully supported, symmetric, finite entropy probability measure $\mu$ on $G$ such that $(G,\mu)$ is not Liouville. In particular, if $G$ is finitely generated, then it is Choquet-Deny if and only if it is virtually nilpotent.
\end{maintheorem}
That a group with no ICC quotients is Choquet-Deny was shown by Jaworski~\cite[Theorem 4.8]{jaworski2004countable}.\footnote{In fact, Jaworski proves there a stronger statement; see the discussion in \S\ref{sec:CD-defs}.} Our contribution is therefore in the proof of the converse, which appears in \S\ref{sec:hard}.

Groups with no ICC quotients are known as FC-hypercentral  (see, e.g.,~\cite{mclain1956remarks,duguid1956fc}, or \cite[\S 4.3]{robinson1972finiteness}).  This class is closed under
forming subgroups, quotients, direct products and finite index extensions, and includes all virtually nilpotent  groups.  Among finitely generated groups, virtually nilpotent groups are precisely those with no ICC quotients (see \cite[Theorem 2]{mclain1956remarks} and \cite[Theorem 2]{duguid1956fc}); this implies the result in Theorem~\ref{thm:main} for finitely generated groups. Since finitely generated groups of exponential growth are not virtually nilpotent, Theorem~\ref{thm:main} implies that the above mentioned conjecture of Kaimanovich and Vershik~\cite{kaimanovich1983random} is correct. 

A very recent result by three of the authors of this paper shows that a countable discrete group is strongly amenable if and only if it has no ICC quotients~\cite{frisch2018strong}. This implies that $G$ is strongly amenable if and only if $(G,\mu)$ is Liouville for every non-degenerate $\mu$, paralleling the above mentioned characterization of amenability as equivalent to the existence of a non-degenerate $\mu$ such that $(G,\mu)$ is Liouville. While the proofs of these two similar results are different, it is natural to ask whether there is some deeper connection between strong amenability and the Choquet-Deny property.

\subsection{Different possible definitions of Choquet-Deny groups}
\label{sec:CD-defs}

Our definition of Choquet-Deny groups is not the usual one, which states that a group is Choquet-Deny if $(G,\mu)$ is Liouville for every {\em adapted} measure $\mu$, where $\mu$ is called adapted if its support generates $G$ as a {\em group} (rather than as a semigroup, as in the non-degenerate case) \cite{guivarc1973croissance, glasner1976proximal, glasner1976choquet}. Yet another definition used in the literature requires that for {\em every} $\mu$, every bounded $\mu$-harmonic function is constant on the left cosets of $G_{\mu}$, where $G_\mu$ is the subgroup of $G$ generated by the support of $\mu$~\cite{jaworski2007choquet}. 

While a priori these are different definitions, they are equivalent, as demonstrated by our result and by Jaworski's Theorem 4.8 in~\cite{jaworski2004countable}. Jaworski's result shows that groups with no ICC quotients are Choquet-Deny according to any of these definitions. Since our construction of $\mu$ with a non-trivial boundary yields measures that are supported on all of $G$ (hence non-degenerate, hence adapted), it shows that groups with ICC quotients are not Choquet-Deny according to any of these definitions. Moreover, our result shows that the class of Choquet-Deny groups (whether defined with adapted or with non-degenerate measures) is closed under taking subgroups, which,  to the best of our knowledge, was also not previously known.

\subsection*{Acknowledgments}
We would like to thank Anna Erschler  and Vadim Kaimanovich for many useful comments on the first draft of this paper. We thank Wojciech Jaworski for bringing a number of errors to our attention and suggesting many improvements. We likewise thank an anonymous referee for many helpful suggestions.

\section{Proofs}

\label{sec:hard}
 In this section we prove the main result of our paper, Theorem~\ref{thm:main}. Unless stated otherwise, we will assume that all groups are countable and discrete. 
 
 Recall that a probability measure $\mu$ on $G$ is symmetric if $\mu(g)=\mu(g^{-1})$ for all $g \in G$. Its Shannon entropy (or just entropy) is $H(\mu)=-\sum_{g \in G}\mu(g)\log\mu(g)$.
 
 Our Theorem~\ref{thm:main} is a direct consequence of \cite[Theorem 4.8]{jaworski2004countable}, which proves it for the case of groups with no ICC quotients, and of the following proposition, which handles the case of groups with ICC quotients.
\begin{proposition}\label{prop:hard-direction}
 Let $G$ be a group with an ICC quotient. Then there exists a fully-supported, symmetric, finite entropy probability measure $\mu$ on $G$ such that $\Pi(G,\mu)$ is non-trivial.
\end{proposition}

The main technical effort in the proof of Proposition~\ref{prop:hard-direction} is in the proof of the following proposition.
\begin{proposition}
\label{prop:non-invariant}
  Let $G$ be an amenable ICC group. For every $h \in G \setminus \{e\}$ there exists a fully supported, symmetric, finite entropy probability measure $\mu$ such that
  \begin{align}
  \label{eq:mu}
      \lim_{m\to\infty}\norm{h\mu^{*m}-\mu^{*m}} > 0.
  \end{align}
\end{proposition}
Here $\mu^{*m}$ is the $m$-fold convolution $\mu * \cdots * \mu$.
We will prove this Proposition later, and now turn to the proof of Proposition~\ref{prop:hard-direction}.

\begin{proof}[Proof of Proposition~\ref{prop:hard-direction}]
  The case of non-amenable $G$ is known, so assume that $G$ is amenable and has an ICC quotient $Q$. Let $h$ be a non-identity element of $Q$. Applying Proposition~\ref{prop:non-invariant} to $Q$ and $h$ yields a finite entropy, symmetric  measure $\bar{\mu}$ on $Q$ that is fully supported, and satisfies~\eqref{eq:mu}. 
  
  Since $\bar{\mu}$ has full support and satisfies~\eqref{eq:mu}, it follows from \cite[Theorem 2]{glasner1976choquet} that $(Q,\bar{\mu})$ has a non-trivial Poisson boundary. Let $\mu$ be any symmetric, finite entropy non-degenerate probability measure on $G$ that is projected to $\bar{\mu}$; the existence of such a $\mu$ is straightforward. Then $(G,\mu)$ has a non-trivial Poisson boundary. 
\end{proof}

\subsection{Switching Elements}\label{sec:switching} Here we introduce two notions: switching elements and super-switching elements. We will use these notions in the proof of Proposition~\ref{prop:non-invariant}.

\begin{definition}
  Let $X$ be a finite symmetric subset of a group $G$.
  \begin{itemize}
    \item We call $g\in G$ a {\em switching element for $X$} if 
    $$X \cap g X g^{-1}  \subseteq \{e\}.$$
    \item We call $g\in G$ a {\em super-switching element for $X$} if
    $$ X \cap \big(g Xg\cup g Xg^{-1} \cup g^{-1} Xg \cup g^{-1} Xg^{-1}\big) \subseteq \{e\}.
    $$
  \end{itemize}
\end{definition}
\bigskip
Note that since $X$ is symmetric, $g\in G$ is a switching element for $X$ if and only if $g^{-1}$ is a switching element for $X$. 

\begin{claim}
\label{claim:super-switching}
  Let $X$ be a finite symmetric subset of a group $G$ and let $g\in G$ be a super-switching element for $X$. If $g^{w_1} x g^{w_2} = y$ for $x,y\in X$ and $w_1, w_2\in \{-1, +1\}$, then $x=y=e$.
\end{claim}
\begin{proof}
  Let $g^{w_1} x g^{w_2} = y$ for $x,y\in X$ and $w_1, w_2\in \{-1, +1\}$. Since 
  $$y=g^{w_1} x g^{w_2} \in \big(g Xg\cup g Xg^{-1} \cup g^{-1} Xg \cup g^{-1} Xg^{-1}\big)$$
  and $y\in X$, it follows from the definition of a super-switching element for $X$ that $y=e$.
  
  From $g^{w_1} x g^{w_2} = y$, we get $g^{-w_1} y g^{-w_2} = x$. So, by symmetry, the same argument shows $x=e$.
\end{proof}

\begin{proposition}
\label{prop:switching-element}
  Let $G$ be a discrete (not necessarily countable) amenable ICC group, and let $X$ be a finite symmetric subset of $G$. The set of super-switching elements for $X$ is infinite.
\end{proposition}

\begin{proof}[Proof of Proposition~\ref{prop:switching-element}]
  Fix an invariant finitely additive probability measure $d$ on $G$. For $A \subseteq G$, we call $d(A)$ the density of $A$. We will need the fact that infinite index subgroups have zero density, and that $d(A) = 0$ for every finite subset $A \subset G$.

  Let $\cent{x}$ be the centralizer of a non-identity $x\in X$. Then, since $X$ is finite, there is a finite set of cosets of $\cent{x}$ that includes all $g \in G$ such that $g^{-1}xg\in X$. So, non-switching elements for $X$ are in the union of finitely many cosets of subgroups with infinite index, since $G$ is ICC. This means that the set of non-switching elements for $X$ has zero density, and so the set $S$ of switching elements for $X$ has density one.
  
  Let $T$ be the set of all super-switching elements for $X$.  Let $A \subseteq G$ be the set of involutions $\{g\in G \ | \ g^2=e\}$. 
  
  If $d(A) > 0$, then $d(A\cap S)>0$. On the other hand, for any $g\in A\cap S$, since $g$ is switching for $X$ and $g^{-1}=g$, $g$ is super-switching for $X$. Hence $A\cap S \subseteq T$. This shows that if $d(A)>0$, then $d(T) \geq d(A\cap S) > 0$, and so we are done.
  
  So, we can assume that $d(A)=0$. For any $x, y\in X$, let $S_{x,y} = \{ g\in S\ | \ gxg = y\}$. Note that
  $$T = S \setminus \bigcup_{\substack{x, y\in X\\ (x,y)\neq(e,e)}} S_{x,y}.$$
  It is thus enough to be shown that each $S_{x,y}$ has zero density when $(x,y)\neq(e,e)$. So assume for the sake of contradiction that $d(S_{x,y})>0$. Fix $g\in S_{x,y}$. We have the following for all $h\in g^{-1}S_{x,y}$.
    \begin{align*}
      gxg=y=gh x gh\implies&\ (xg)= h(xg)h\\
      \implies&\ (xg)^{-1} h^{-1} (xg) = h\\
      \implies&\ h = (xg)^{-1} h^{-1} (xg)\\
      &\quad = (xg)^{-1} [(xg)^{-1} h^{-1} (xg)]^{-1} (xg)\\
      &\quad = (xg)^{-2} h (xg)^2\\
      \implies&\ \text{$h$ is in the centralizer of $(xg)^2$.}
    \end{align*}
    So, the centralizer of $(xg)^2$ includes $g^{-1} S_{x,y}$, which has a positive density. So, the centralizer of $(xg)^2$ has finite index. This implies that $(xg)^2 = e$, because in an ICC group only the identity can have a finite index
centralizer. Hence $x g \in A$ for all $g \in S_{x,y}$. So $x S_{x,y} \subseteq A$. Hence $S_{x,y}$ also has zero density, which is a contradiction.
\end{proof}

\subsection{A Heavy-Tailed Probability Distribution on $\N$.}\label{sec:heavy-tailed} Here we state and prove a lemma about the existence of a probability distribution on $\N=\{1,2,\ldots\}$ such that infinite i.i.d.\ samples from this measure have certain properties. We will use this distribution in the proof of Proposition~\ref{prop:non-invariant}.

\begin{lemma}
\label{lemma:random-string}
Let $p$ be the following probability measure on $\N$:
  $p(n) = cn^{-5/4}$, where $1/c=\sum_{n=1}^\infty n^{-5/4}$.
  Then $p$ has finite entropy  and the following property: for any $\eps>0$ there exist constants $K_\eps,N_\eps\in\N$ such that for any natural number $m\geq K_\eps$ there exists an $E_{\eps,m}\subseteq \N^m$ such that:
  \begin{enumerate}
    \item $p^{\times m}(E_{\eps,m}) \geq 1-\eps$, where $p^{\times m}$ is the $m$-fold product measure $p \times \cdots \times p$.
    \item For any $s=(s_1,\ldots,s_m)\in E_{\eps,m}$, the maximum of $\{s_1,\ldots,s_{K_\eps}\}$ is at most $N_\eps$.
    \item For any $s=(s_1,\ldots,s_m)\in E_{\eps,m}$ and for any $K_\eps\leq k\leq m$, the maximum of $\{s_1,\ldots,s_k\}$ is at least $k^2$.
    \item For any $s=(s_1,\ldots,s_m)\in E_{\eps,m}$ and for any $K_\eps\leq k\leq m$, the maximum of $\{s_1,\ldots,s_k\}$ appears in $(s_1,\ldots,s_k)$ only once.
  \end{enumerate}
\end{lemma}
\begin{proof}
   It is
  straightforward to see that $p$ has finite entropy.
  
  Let $s=(s_1,s_2,\ldots)\in \N^\infty$ have distribution
  $p^{\times \infty}$; i.e., $s$ is a sequence of i.d.d.\ random variables with distribution $p$. Since each
  $s_i$ has distribution $p$, for each $n\in\N$ we have:
  \begin{align}
  \label{eq:p}
    \Pr{s_i \geq n} = \sum_{m=n}^\infty p(m) = c \sum_{m=n}^\infty m^{-5/4} \geq c \int_{n}^\infty x^{-5/4} \dd x = 4c n^{-1/4}.
  \end{align}
  
  For $k\geq 1$, let
  \begin{align*}
    M_k \coloneqq & \max\{s_1,\ldots,s_k\},
  \end{align*}
  and let
  \begin{align*}
    \text{next}(k) \coloneqq & \min\{i > k \ | \ s_i \geq M_k \}.
  \end{align*}
  In words, $\text{next}(k)$ is the first index $i > k$ for which
  $s_i$ matches or exceeds $M_k$.

  We first show that with probability one, $M_k \geq k^2$
  for all $k$ large enough. To this end, let $A_k$ be the event that
  $M_k < k^2$.  We have:
  \begin{align*}
    \Pr{A_k} = & \ \Pr{s_i < k^2 \ \forall i\in\{1,\ldots,k\}}\\
    =& \ (1-\Pr{s_1 < k^2})^k\\
    \leq & \ (1-4c(k^2)^{-1/4})^k\\
    \leq& \ e^{-4ck^{1/2}}.
  \end{align*}
  Since the sum of these probabilities is finite, by Borel-Cantelli we get that
  \begin{align*}
    \Pr{A_k\ \text{infinitely often}} = 0.
  \end{align*}
  Hence $M_k \geq k^2$ for all $k$ large enough, almost
  surely. Furthermore, the expectation of $1/M_k$ is small:
  \begin{align}
    \label{eq:emk}
    \E{\frac{1}{M_k}}
    =
    \E{\frac{1}{M_k}}{A_k}\Pr{A_k} +
      \E{\frac{1}{M_k}}{\neg A_k}\Pr{\neg A_k}  
    \leq e^{-4ck^{1/2}} + \frac{1}{k^2}\cdot
  \end{align}

  Next, we show that, with probability one, $s_{\text{next}(k)} >
  M_k$ for all $k$ large enough. That is, for large enough
  $k$, the first time that $M_k$ is matched or exceeded
  after index $k$, it is in fact exceeded.

  Let $B_k$ be the event that $s_{\text{next}(k)} = M_k$. We
  would like to show that this occurs only finitely often. Note that
  \begin{align*}
    \Pr{B_k}{M_k}
    &= \Pr{s_{\text{next}(k)}=M_k}{M_k}\\
    &= \sum_{i=k+1}^\infty\Pr{s_i=M_k, \text{next}(k)=i}{M_k}.
  \end{align*}
  Applying the definition of $\text{next}(k)$ yields
  \begin{align*}
    \Pr{B_k}{M_k}
    &= \sum_{i=k+1}^\infty\Pr{s_i=M_k, s_{k+1},\ldots,s_{i-1}<M_k}{M_k}.
  \end{align*}
  By the independence of the $s_i$'s we can write this as
  \begin{align*}
    \Pr{B_k}{M_k}
    &= \sum_{i=k+1}^\infty\Pr{s_i=M_k}{M_k}\prod_{n=1}^{i-(k+1)}\Pr{s_{k+n}<M_k}{M_k}\\
    &= \sum_{i=k+1}^\infty\frac{c}{M_k^{5/4}}\Pr{s_{k+1}<M_k}{M_k}^{i-(k+1)}.
  \end{align*}
  By~\eqref{eq:p}, $\Pr{s_{k+1}<M_k}{M_k} \leq
  1-4 c M_k^{-1/4}$. Hence
  \begin{align*}
    \Pr{B_k}{M_k}
    \leq \frac{c}{M_k^{5/4}} \cdot \frac{1}{4 c M_k^{-1/4}}
    =\frac{1}{4 M_k}\cdot 
  \end{align*}
  Using~\eqref{eq:emk} it follows that
  \begin{align*}
    \Pr{B_k} = \E{\Pr{B_k}{M_k}} \leq \E{\frac{1}{4M_k}} \leq \frac{1}{4}e^{-4ck^{1/2}} + \frac{1}{4k^2}.
  \end{align*}
  Hence $\sum_k \Pr{B_k} < \infty$, and so by Borel-Cantelli $B_k$
  occurs only finitely often.

  Since $A_k$ and $B_k$ both occur for only finitely many $k$, the (random) index $\text{ind}'$ at which they stop occurring is almost surely finite, and is given by
  $$
    \text{ind}' = \min\{\ell \in \N \,:\, s \not \in A_k \cup B_k \text{ for all } k \geq \ell\}.
  $$ 
  Let
  $$\text{ind} = \text{next}(\text{ind}').$$ 
  Hence for $k \geq \text{ind}$, $M_k \geq k^2$ and $M_k$ appears in $(s_1,\ldots,s_k)$ only once.
  
  Fix $\eps > 0$. Since $\text{ind}$ is almost surely finite, then for large enough constants $K_\eps\in\N$ and $N_\eps \in \N$ the event $$E_\eps = \{\text{ind} \leq K_\eps\text{ and } M_{K_\eps}\leq N_\eps\}$$ has probability at least $1-\eps$, and additionally, conditioned on $E_\eps$ it holds that
  $k \geq \text{ind}$ for all $k \geq K_\eps$, and hence $M_k \geq k^2$ and $M_k$ appears in $(s_1,\ldots,s_k)$ only once. Therefore, if for $m \geq K_\eps$ we let $E_{\eps,m}$ be the projection of $E_\eps$ to the first $m$ coordinates, then $E_{\eps,m}$ satisfies the desired properties.
\end{proof}

\subsection{Proof of Proposition~\ref{prop:non-invariant}}
  Let $\frac{1}{8}>\eps>0$. Let $p$, $K_\eps\in \N,N_\eps \in \N$, and $E_{\eps,m}\subseteq \N^m$ be the probability measure, the constants, and the events from Lemma~\ref{lemma:random-string}. To simplify notation let $N=N_\eps$ and $K=K_\eps$. 
  
  Let $G = \{a_1,a_2,\ldots\}$, where $a_1=a_2=\cdots=a_N=e$. We define $(g_n)_{n}$, $(A_n)_{n}$, $(B_n)_n$ and $(C_n)_n$ recursively. Given $g_1,\ldots,g_{n}$, let $A_n = \{g_n, g_n^{-1},a_n, a_n^{-1}\}$ and $B_n=\cup_{i\leq n} A_i$. Denote $C_n = B_n\cup\{h^{-1},h\}$.  Note that $A_n$, $B_n$, and $C_n$ are finite and symmetric for any $n\in \N$. Let $g_1=g_2=\ldots=g_N=e$. For $n +1 >  N$, given $C_n$, let  $g_{n+1} \in G$ be a super-switching element for $(C_n)^{2n+1}$ which is not in $(C_n)^{8n+1}$. The existence of such a super-switching element is guaranteed by Proposition~\ref{prop:switching-element} and the facts that $(C_n)^{2n+1}$ is a finite symmetric subset of $G$ and that $(C_n)^{8n+1}$ is finite.

  
  
  For $n \in \N$, define a symmetric probability measure $\mu_n$ on $A_n$ by
  $$\mu_n = \eps2^{-n} (\frac{1}{2}\delta_{a_n} + \frac{1}{2} \delta_{a_n^{-1}}) + (1-\eps2^{-n}) (\frac{1}{2}\delta_{g_n} + \frac{1}{2}\delta_{g_n^{-1}}).$$
  Here $\delta_g$ is the point mass on $g \in G$.
  Finally, let 
  $$\mu = \sum_{n=1}^{\infty} p(n) \mu_n.$$ 
  Obviously $\mu$ is symmetric and $\supp \mu=G$. Since $p$ has finite entropy and each $\mu_n$ has support of size at most 4, it follows easily that $\mu$ has finite entropy. 
   
  We want to show that 
  $$
  \lim_{m\to\infty}\norm{h\mu^{*m}- \mu^{*m}} > 0.
  $$
  
  Fix $m\in \N$ larger than $K$ and $N$. For each $n \in \N$ define $f_n:\{1,2,3,4\}\to A_n$ by
  \begin{align*}
    f_n(1) = a_n,\
    f_n(2) = a_n^{-1},\
    f_n(3) = g_n,\
    f_n(4) = g_n^{-1},
  \end{align*}
  and define $\nu_n:\{1,2,3,4\}\to [0,1]$ by
  \begin{align*}
    \nu_n(1) = \nu_n(2) = \frac{1}{2} \eps2^{-n},\
    \nu_n(3) = \nu_n(4) = \frac{1}{2} (1-\eps2^{-n}).
  \end{align*}
  Let
  $$\Omega = \{(s,w)\ | \ s\in\N^m, \ w\in\{1,2,3,4\}^m \}.$$

  We define the measure $\eta$ on the countable set $\Omega$ by specifying its values on the singletons:
  $$\eta(\{(s,w)\}) = p^{\times m}(s)\ \nu_{s_1}(w_1)\ \nu_{s_2}(w_2)\ \ldots\ \nu_{s_m}(w_m).$$
  It follows immediately from this definition that $\eta$ is a probability measure.
  
  Define $r:\Omega \to G$ by
  $$r(s,w) = f_{s_1}(w_1) f_{s_2}(w_2) \ldots f_{s_m}(w_m).$$
  It is not difficult to see that $r_{*}\eta = \mu^{*m}$, and so we need to show that $\norm{h r_*\eta - r_*\eta}$ is uniformly bounded away from zero for $m$ larger than $K$ and $N$.
  
  Recall that $E_{\eps,m}\subseteq \N^m$ is the event given by Lemma~\ref{lemma:random-string}. Fix $s\in E_{\eps,m}$. Define
  \begin{align*}
    i_{s,1} =&\ \min\{j\in \{1\ldots,m\} \ | \ s_j > N\},\\
    i_{s,2} =&\ \min\{j > i_{s,1} \ | \ s_j \geq s_{i_{s,1}}\},\\
    & \vdots \\
    i_{s,l(s)} =&\ \min\{j > i_{s,l(s)-1} \ | \ s_j \geq s_{i_{s,l(s)-1}} \}.
  \end{align*}
  Note that by the second property of $E_{\eps,m}$ in Lemma~\ref{lemma:random-string}, we know that
  \begin{align*}
    K < i_{s,1} < i_{s,2} < \cdots < i_{s,l(s)},
  \end{align*}
  and by the fourth property,
  \begin{align*}
    N< s_{i_{s,1}} < s_{i_{s,2}} < \cdots < s_{i_{s,l(s)}} = \max\{ s_1,\ldots, s_m\}.
  \end{align*}
  Let 
  $$W_\eps^{s}=\{w\in\{1,2,3,4\}^m \ | \ \forall k\leq l(s) \ w_{i_{s,k}}=3,4 \}.$$ 
  
  For $s \in \N^m$ let $\eta_s$ be the measure $\eta$, conditioned on the first coordinate equalling $s$. I.e., let
    $$
  \eta_s(A) = \frac{\eta(A \cap \Omega^{s})}{\eta(\Omega^{s})},
  $$
  where $\Omega^{s} = \{s\}\times\{1,2,3,4\}^m \subseteq \Omega$.

  Then
  \begin{align*}
    \eta_s(\{s\}\times W_\eps^{s}) 
    =&\ 1 - \eta_s(\{w_{i_{s,1}}=1,2; \text{ or } w_{i_{s,2}}=1,2;\ \ldots; \text{ or } w_{i_{s,l(s)}}=1,2\ \})\\
    \geq&\ 1 - \sum_{k=1}^{l(s)} \eta_s(\{w_{i_{s,k}}=1,2\})\\
    =&\ 1 - \sum_{k=1}^{l(s)} \eps2^{-s_{i_{s,k}}}\\
    \geq&\ 1 - \sum_{j=1}^\infty \eps2^{-j}\\
    =&\ 1-\eps,
  \end{align*}
  where the first inequality follows from the union bound, and the last inequality holds since $s_{i_{s,1}} < s_{i_{s,2}} < \cdots < s_{i_{s, l(s)}}$.
  
  Finally, let
  $$\Omega_\eps = \{(s,w)\in\Omega \ | \ s\in   E_{\eps,m}, \ w\in W_\eps^{s} \}.$$
  By the above, and since $\eta(E_{\eps,m}\times\{1,2,3,4\}^m) \geq 1-\eps$ by Lemma~\ref{lemma:random-string}, we have shown that 
  \begin{align*}
    \eta(\Omega_\eps)\geq\ (1-\eps)(1-\eps)>\ 1-2\eps.
  \end{align*}
  
  \begin{claim}
  \label{claim:E_eps}
    For any $\alpha, \beta\in \Omega_\eps$, we have $h r(\alpha) \neq r(\beta)$.
  \end{claim}
  We prove this claim after we finish the proof of the Proposition.
  
  Let $\eta_1$ be equal to $\eta$ conditioned on $\Omega_\eps$, and $\eta_2$ be equal to $\eta$ conditioned on the complement of $\Omega_\eps$. We have $\eta = \eta(\Omega_\eps) \eta_1 + (1-\eta(\Omega_\eps)) \eta_2$, and by the above claim we know $\norm{hr_*\eta_1 - r_*\eta_1} = 2$. So for $m$ larger than $K$ and $N$
  \begin{align*}
    \norm{h\mu^{*m}- \mu^{*m}} =&\ \norm{hr_*\eta - r_*\eta}\\
    =&\ \norm{\eta(\Omega_\eps)(h r_*\eta_1 - r_*\eta_1) + (1-\eta(\Omega_\eps))(h r_*\eta_2 - r_*\eta_2)}\\
    \geq&\ \eta(\Omega_\eps) \norm{h r_*\eta_1 - r_*\eta_1} - 2 (1-\eta(\Omega_\eps))\\
    \geq&\ 2(1-2\eps) -2(2\eps)= 2-8\eps,
  \end{align*}
  which is uniformly bounded away from zero since $\eps < \frac{1}{8}$. Since $\norm{h\mu^{*m}- \mu^{*m}}$ is a decreasing sequence,
  this completes the proof of Proposition~\ref{prop:non-invariant}.

\begin{proof}[Proof of Claim~\ref{claim:E_eps}]
  Let $\alpha=(s,w),\ \beta=(t,v)\in \Omega_\eps$. Hence $\max\{K,N\}< m$, $s\in E_{\eps,m}$, $t\in E_{\eps,m}$, $w\in W_\eps^{s}$, and $v\in W_\eps^{t}$. Assume that $hr(\alpha) = r(\beta)$. So, we have
  $$hf_{s_1}(w_1)\cdots f_{s_m}(w_m) = f_{t_1}(v_1)\cdots f_{t_m}(v_m).$$
  Let $K<i_1<i_2<\cdots<i_{l(s)}$ and $K<j_1<j_2<\cdots<j_{l(t)}$ be the indices we defined for $s$ and $t$ in the proof of Proposition~\ref{prop:non-invariant}. We remind the reader that the unique maximum of $(s_1,\ldots,s_m)$ is attained at $i_{l(s)}$, with a corresponding statement for $(t_1,\ldots,t_m)$ and $j_{l(t)}$. So we have
  \begin{align*}
    h &\overbrace{f_{s_1}(w_1)\cdots f_{s_{i_{l(s)}-1}}(w_{i_{l(s)}-1})}^{b_1} f_{s_{i_{l(s)}}}(w_{i_{l(s)}}) \overbrace{f_{s_{i_{l(s)}+1}}(w_{i_{l(s)}+1})\cdots f_{s_m}(w_m)}^{b_2}\\
    =&\ \underbrace{f_{t_1}(v_1)\cdots f_{t_{j_{l(t)}-1}}(v_{j_{l(t)}-1})}_{c_1} f_{t_{j_{l(t)}}}(v_{j_{l(t)}}) \underbrace{f_{t_{j_{l(t)}+1}}(v_{j_{l(t)}+1})\cdots f_{t_m}(v_m)}_{c_2}.
  \end{align*}
  Let $p = s_{i_{l(s)}} = \max\{s_1,\ldots,s_m\}$ and $q = t_{j_{l(t)}} = \max\{t_1,\ldots,t_m\}$. Since $w\in W_\eps^{s}$ and $v\in W_\eps^{t}$, we know $f_{s_{i_{l(s)}}}(w_{i_{l(s)}}) = g_p^{\pm 1}$ and $f_{t_{j_{l(t)}}}(v_{j_{l(t)}}) = g_q^{\pm 1}$, so 
  \begin{align}
      \label{eq:hrr}
      h b_1 g_p^{\pm 1} b_2 = c_1 g_q^{\pm 1} c_2.
  \end{align}
  Since $p = \max\{s_1,\ldots,s_m\}$, and since $m\geq K$, we know that $m\leq m^2 \leq p$. So $b_1,b_2\in (B_{p-1})^{p-1}\subseteq (C_{p-1})^{p-1}$. Similarly $c_1,c_2\in (C_{q-1})^{q-1}$.
  
  Consider the case that $p>q$. Then $c_1,c_2,g_q^{\pm 1}\in (C_q)^q\subseteq (C_{p-1})^{p-1}$. Hence  $g_p^{\pm 1} = [b_1^{-1}] h^{-1} [c_1 g_q^{\pm 1} c_2 b_2^{-1}]$ by~\eqref{eq:hrr}, and so $$g_p \in (C_{p-1})^{4(p-1)} \{h,h^{-1}\} (C_{p-1})^{4(p-1)}\subseteq (C_{p-1})^{8(p-1)+1},$$ which is a contradiction with our choice of $g_p$, since $p>N$. Similarly, if $p<q$, we get a contradiction. So we can assume that $p=q$.
  
  If $p = q$, then by~\eqref{eq:hrr} we have 
  $$h b_1 g_p^{\pm 1} b_2 = c_1 g_p^{\pm 1} c_2,$$
  and $c_1, c_2, b_1, b_2 \in (C_{p-1})^{p-1}$. So, for $x=c_1^{-1} h b_1\in (C_{p-1})^{2(p-1)+1}$ we have $g_p^{\pm 1} x g_p^{\pm 1} = c_2 b_2^{-1} \in (C_{p-1})^{2(p-1)}\subseteq (C_{p-1})^{2(p-1)+1}$. By the fact that $g_p$ is a super-switching element for $(C_{p-1})^{2(p-1)+1}$ and from Claim ~\ref{claim:super-switching}, we get that $x$ is the identity.
  
  So $h b_1 = c_1$, i.e.
  $$hf_{s_1}(w_1)\cdots f_{s_{i_{l(s)}-1}}(w_{i_{l(s)}-1}) = f_{t_1}(v_1)\cdots f_{t_{j_{l(t)}-1}}(v_{j_{l(t)}-1}).$$
  By the exact same argument, we can see this leads to a contradiction unless
  $$hf_{s_1}(w_1)\cdots f_{s_{i_{l(s)-1}-1}}(w_{i_{l(s)-1}-1}) = f_{t_1}(v_1)\cdots f_{t_{j_{l(t)-1}-1}}(v_{j_{l(t)-1}-1}).$$
  And again, this leads to a contradiction unless
  $$hf_{s_1}(w_1)\cdots f_{s_{i_{l(s)-2}-1}}(w_{i_{l(s)-2}-1}) = f_{t_1}(v_1)\cdots f_{t_{j_{l(t)-2}-1}}(v_{j_{l(t)-2}-1}).$$
  Note that if $l(s)\neq l(t)$, at some point in this process we get that either all the $s_i$'s or all the $t_i$'s are at most $N$ while the other string has characters strictly greater than $N$. This leads to a contradiction similar to the case $p\neq q$, which we explained before. So, by continuing this process, we get a contradiction unless
  \begin{align}
      \label{eq:last step}
      h f_{s_1}(w_1)\cdots f_{s_{i_1-1}}(w_{i_1-1}) = f_{t_1}(v_1)\cdots f_{t_{j_1-1}}(v_{j_1-1}).
  \end{align}
  Note that $s_1,\ldots,s_{i_1-1}\leq N$, which implies $$f_{s_1}(w_1)=\cdots=f_{s_{i_1-1}}(w_{i_1-1})=e.$$
  Similarly,  $t_1,\ldots,t_{j_1-1}\leq N$ implies that  $$f_{t_1}(v_1)=\cdots=f_{t_{j_1-1}}(v_{j_1-1})=e.$$
  So, from ~\eqref{eq:last step} we get $h=e$, which is a contradiction.

\end{proof}

\bibliography{main}
\end{document}